\documentclass[preprint,10pt]{elsarticle}
\usepackage{amsmath}
\usepackage{float}
\usepackage{algorithm}
\usepackage{algpseudocode}
\usepackage{amsthm}
\usepackage{subfigure}
\usepackage{txfonts}
\usepackage{booktabs}
\usepackage{algorithm}
\usepackage{natbib}
\usepackage{appendix}
\usepackage{color}
\usepackage{enumitem}
\usepackage{bibspacing}
\allowdisplaybreaks[4]
\biboptions{numbers,sort&compress}
\usepackage[colorlinks=true]{hyperref}
\usepackage[margin=0.75in]{geometry}
\numberwithin{figure}{section}
\makeatletter
\@addtoreset{algorithm}{section}
\makeatother
\newtheorem{thm}{Theorem}[section]
\newtheorem{lem}{Lemma}[section]
\newtheorem{cor}[thm]{Corollary}
\newtheorem{exm}{Example}[section]
\newtheorem{asm}{Assumption}[section]

\numberwithin{equation}{section}

\numberwithin{table}{section}

\begin{document}
	\begin{frontmatter}
		\title{A stochastic column-block gradient descent method for solving nonlinear systems of equations\footnote{}}
		\date{today}
		\author{Naiyu Jiang}
		\author{Wendi Bao}
		\ead{baowendi@sina.com}
		\author{Lili Xing}
		\author{Weiguo Li}
		\address{College of Science,
			China University of Petroleum, Qingdao 266580, P.R. China}	
		
		\begin{abstract}
			In this paper, we propose a new stochastic column-block gradient descent method for solving nonlinear systems of equations. It has a descent direction and holds an approximately optimal step size obtained through an optimization problem. We provide a thorough convergence analysis, and derive an upper bound for the convergence rate of the new method. Numerical experiments demonstrate that the proposed method outperfroms the existing ones.
		\end{abstract}
		\begin{keyword}
			Nonlinear equations, Gradient descent method, Gauss-Seidel method, Block method.
		\end{keyword}
	\end{frontmatter}
	
	\section{Introduction}
	\label{sec1}
	
In this paper, we consider the nonlinear system of  equations
\begin{equation}\label{eq1.1}
	f(x)=0,
\end{equation}
where $f(x):\mathcal{D}\subseteq \mathbb{R}^n\to \mathbb{R}^m$ is a continuously differentiable function, $x\in \mathbb{R}^n$ is the variable and \eqref{eq1.1} is solvable, and some solution is denoted by $x_*$. Nonlinear equations are omnipresent and play crucial roles in scientific computing and mathematical modeling, ranging from differential equations, integral equations, and optimization problems to data-driven modeling and machine learning.
	 
	 It is clear that (\ref{eq1.1}) is equivalent to the nonlinear least-squares problem 
	\begin{equation}\label{eq1.2}
		\mathop{\min}_{x\in\mathcal{D}}\frac{1}{2}\|f(x)\|_2^2.
	\end{equation}
Newton method and Newton-like methods (Gauss Newton \cite{Blaschke97}, inexact Newton \cite{Dembo82}, Newton-Krylov methods \cite{Knoll04}) are famous methods to solve the above nonlinear least squares problem.	However, the expensive computational cost still limits the application of Newton’s method to some practical applications, such as data-driven modeling problems and optimization problems arising from machine learning. To accelerate the convergence of existing methods,	Hao \cite{Hao21} proposed the gradient descent (GD) method and the stochastic gradient descent (SGD) method, which is the stochastic row-block version of the GD method. They theoretically prove that the GD method has linear convergence in general. The numerical examples show that the two methods are more efficient than the Newton method. Motivated by the doubly stochastic block Gauss-Seidel (DSBGS) method, Bao \cite{Bao25} proposed the doubly stochastic block method for nonlinear equations (DSBN) method. Furthermore, to reduce the computations of the Jacobi matrix in random probabilities, a stochastic column-block method (SCBN) with uniform probabilities is established for nonlinear equations.

	    Inspired by \cite{Hao21} and \cite{Bao25}, we propose a new stochastic column-block gradient descent (SCBGD) method with an approximately optimal step size for solving nonlinear equations. We verify that the new method has a descent direction and derive its convergence.
	Numerical experiments are presented to demonstrate the efficiency of the new method.
	
	The notations in this paper are expressed as follows. For any matrix $A\in R^{m\times n}$, $\| \cdot\|_2$, $\sigma_{\min}$, $\sigma_{\max}$ denote the Euclidean norm, the maximum and minimum nonzero singular values of a matrix $A$, respectively. Let $\nabla f(x)$ be the Jacobi matrix of $f(x)$ and $\nabla f_{:,j}(x)$ is the $j$th column of the Jacobi matrix $\nabla f(x)$.
		
	The rest of this paper is organized as follows. In Section \ref{sec2}, we propose the SCBGD method and analyze its convergence. We perform some numerical experiments to verify the efficiency of the SCBGD method in Section \ref{sec3}. Finally, we give the conclusions in Section \ref{sec4}.
	
	\section{Stochastic column-block gradient descent method}
	\label{sec2}
	First, we recall the gradient descent (GD) method \cite{Hao21}. The iterative format of GD method for solving the optimization problem (\ref{eq1.2}) is 
	\begin{equation*}
		x_{k+1}=x_k-\eta_k\nabla f(x_k)^Tf(x_k),
	\end{equation*} 
	where $\eta_k$ is the step size. Using Taylor expansion, we have $f(x_{k+1})\approx f(x_k)-\eta_k\nabla f(x_k)\nabla f(x_k)^Tf(x_k)\approx0$, which implies that $\eta_k=\frac{v^Tf(x_k)}{v^T\nabla f(x_k)\nabla f(x_k)^Tf(x_k)}$ for any given $v$. Choosing $v=\nabla f(x_k)\nabla f(x_k)^Tf(x_k)$ can let $f(x_{k+1})^Tf(x_{k+1})\leq f(x_k)^Tf(x_k)$, which guarantees the convergence. Then the iterative format of GD method is written as
	\begin{equation*}
		x_{k+1}=x_k-\frac{\|\nabla f(x_k)^Tf(x_k)\|_2^2}{\|\nabla f(x_k)\nabla f(x_k)^Tf(x_k)\|_2^2}\nabla f(x_k)^Tf(x_k).
	\end{equation*}  
	
	Since GD method needs to calculate the entire Jacobi matrix every iteration, which is a large computation for higher dimensional problems, we consider the column block type of GD method. We take some columns of $\nabla f(x_k)$, then the iterative format can be written as 
	\begin{equation*}
		x_{k+1}=x_k-\eta_kI_{:,\xi_k}\nabla f_{:,\xi_k}(x_k)^Tf(x_k),
	\end{equation*}
	where  $\eta_k$ is the step size and $\xi_k$ is the set of selected columns. Let $p_k=I_{:,\xi_k}\nabla f_{:,\xi_k}(x_k)^Tf(x_k)$. Using Taylor expansion, we have $f(x_{k+1})\approx f(x_k)-\eta_k\nabla f(x_k)p_k\approx0$. Then we obtain
	
	\begin{equation}\label{eq:f^Tf}
		f(x_{k+1})^Tf(x_{k+1})=f(x_k)^Tf(x_k)-2\eta_k\left(\nabla f(x_k)p_k\right)^Tf(x_k)+\eta_k^2\left(\nabla f(x_k)p_k\right)^T\nabla f(x_k)p_k.\\
	\end{equation}
	In order to minimize $f(x_{k+1})^Tf(x_{k+1})$, using the properties of quadratic functions of $\eta_k$, we have $\eta_k=\frac{\left(\nabla f(x_k)p_k\right)^Tf(x_k)}{\left(\nabla f(x_k)p_k\right)^T\nabla f(x_k)p_k}$.
	To make the method more general, we can choose
	\begin{equation}
		\eta_k=\delta\frac{\left(\nabla f(x_k)p_k\right)^Tf(x_k)}{\left(\nabla f(x_k)p_k\right)^T\nabla f(x_k)p_k},		
	\end{equation}
	where $\delta\in(0,2)$.
	
	\begin{lem}
		Let $d_k=-\eta_kI_{:,\xi_k}\nabla f_{:,\xi_k}(x_k)^Tf(x_k)$. Then $d_k$  is a strict descent direction of $g(x)=\frac{1}{2}\|f(x)\|_2^2$, if $\|\nabla f_{:,\xi_k}(x_k)^Tf(x_k)\|_2^2\neq0$.
		\begin{proof}
			It is easy to see
			\begin{align*}
				(-\eta_kI_{:,\xi_k}\nabla f_{:,\xi_k}(x_k)^Tf(x_k))^T\nabla g(x_k)
				&=-\eta_kf(x_k)^T\nabla f_{:,\xi_k}(x_k)I_{\xi_k,:}\nabla f(x_k)^Tf(x_k)\\
				&=-\eta_kf(x_k)^T\nabla f_{:,\xi_k}(x_k)\nabla f_{:,\xi_k}(x_k)^Tf(x_k)\\
				&=-\eta_k\|\nabla f_{:,\xi_k}(x_k)^Tf(x_k)\|_2^2.
			\end{align*}
			Since $\eta_k=\delta\frac{\|\nabla f_{:,\xi_k}(x_k)^Tf(x_k)\|_2^2}{\|\nabla f_{:,\xi_k}(x_k)\nabla f_{:,\xi_k}(x_k)^Tf(x_k)\|_2^2}$ and $\|\nabla f_{:,\xi_k}(x_k)^Tf(x_k)\|_2^2\neq0$, we have $(-\eta_kI_{:,\xi_k}\nabla f_{:,\xi_k}(x_k)^Tf(x_k))^T\nabla g(x_k)<0$.
		\end{proof}
	\end{lem}
	
	Suppose that we take $q$ columns at a time from the Jacobi matrix, then the number of the set of column blocks has $\tau=\left(^n_q\right)$ elements. Let $\{J_1,J_2,\cdots,J_\tau\}$ denotes the set of column blocks. Here we take the uniform probability to choose the selected column block from $\{J_1,J_2,\cdots,J_\tau\}$, then the SCBDG method is derived and described in Algorithm \ref{alg: SCBGD} as follows.
	\begin{algorithm}
		\caption{SCBGD: Stochastic Column-Block Gradient Descent method}\label{alg: SCBGD}
		\begin{algorithmic}[1]
			\State \textbf{parameters:} $\delta\in(0,2)$, $q>0$
			\State \textbf{initialization:} Choose $x_0\in R^n$
			\For {$k=1,2,\cdots $} 
			\State Pick $\xi_k$ uniformly at random from the set $\left\{J_1,J_2,\cdots,J_\tau\right\}$
			\State Compute $x_{k+1}=x_k-\delta\frac{\|\nabla f_{:,\xi_k}(x_k)^Tf(x_k)\|_2^2}{\|\nabla f_{:,\xi_k}(x_k)\nabla f_{:,\xi_k}(x_k)^Tf(x_k)\|_2^2}I_{:,\xi_k}\nabla f_{:,\xi_k}(x_k)^Tf(x_k)$
			\EndFor
			\State \textbf{return:} $x_{k}$
		\end{algorithmic}
	\end{algorithm}
	
	\begin{asm}\label{asm1}
		Let $g(x)=\frac{1}{2}f(x)^Tf(x)$ and a set $Q$ be nonempty and bounded, which attains all minimum value $g(x)$. Assume that the function $g(x)$ is convex and satisfies the following assumptions:
		
		i) The nonlinear system of equations $f(x):\mathcal{D}\subseteq\mathbb{R}^n\to\mathbb{R}^m$ on a bounded closed $\mathcal{D}$ is derivable and the derivative of $f(x)$ is continuous in $\mathcal{D}$.
		
		ii) The level set $\Omega=\left\{x \, | \, g(x)\leq g(x_{0})\right\}$ is bounded. There is a finite $M_0$ such that the level set for $g$ defined by $x_0$ is bounded, that is
		\begin{equation}\label{R0}
			R_0 = \underset{x}{\max}\{\underset{x_*\in Q}{\max}\|x-x^{*}\|_2^2:g(x)\leq g(x_{0})\}.
		\end{equation} 
		
		iii) We define the vector of partial derivatives corresponding to the variables in the vector $x(J_i)=I_{:,J_i}^Tx$ as
		$\nabla_ig(x)=I_{:,\xi_i}^T\nabla g(x)$, $i=1,2,...,\tau$. For $\forall x\in \mathcal{D}(f)$, it holds that
		\begin{equation}\label{eq:g'}
			\|\nabla_ig(x+I_{:,J_i}p_i)-\nabla_ig(x)\|_2\leq L_i\|p_i\|_2 , \forall p_i\in \mathbb{R}^{\hat{\xi}_i},
		\end{equation}
		where $L_i$ is the Lipschitz constant corresponding to block $i$ and $g(x)$ is the block-coordinatewise Lipschitz continuous.
	\end{asm}
	\begin{lem}\label{lem:g(v)}
		\cite{Beck13}. If $g(x)$ satisfies Assumption \ref{asm1} and $i\in\{1,2,...,\tau\}$. Let $u,v\in \mathbb{R}^n$ be two vectors which differ only in the $i$th block, i.e. there exists a $p\in \mathbb{R}^{\hat{J}_i}$ such that $v-u=I_{:,J_i}p$. Then
		\begin{equation*}
			g(v)\leq g(u)+<\nabla g(u),v-u>+\frac{L_i}{2}\|u-v\|_2^2.
		\end{equation*} 
	\end{lem}
	\begin{lem}\label{lem:sigma}
		Assume that Assumption \ref{asm1} i) holds true. Then there exists $\underline{\sigma}$ and $\overline{\sigma}$ such that
		\begin{equation*}
			\underset{x\in\mathcal{D}}{\inf}\ \sigma_{\min}(\nabla f(x))=\underline{\sigma}>0,\
			\underset{x\in\mathcal{D}}{\sup}\ \sigma_{\max}(\nabla f(x))=\overline{\sigma}<\infty.
		\end{equation*} 
		\begin{proof}
			Since $\nabla f(x)$ is continuous in $\mathcal{D}$ and the singular value of $\nabla f(x)$ is a continuous function of $\nabla f(x)$, we can get the singular value of $\nabla f(x)$ is a continuous function of $x$.
			Meanwhile, $\mathcal{D}$ is a bounded closed and $\sigma_{\min}(x_k)$ is continuous function of $x$, there exists a point $\underline{x}\in\mathcal{D}$ such that $0<\underline{\sigma}=\sigma_{\min}(\nabla f(\underline{x}))=\underset{x\in\mathcal{D}}{\inf}\ \sigma_{\min}(\nabla f(x))$. Then for $\forall x_k\in \mathcal{D}$, we have
			\begin{equation*}
				\sigma_{\min}(\nabla f(x_k))\geq\underset{x\in\mathcal{D}}{\inf}\ \sigma_{\min}(\nabla f(x))=\underline{\sigma}>0.
			\end{equation*}
			Similarly there exists a point $\overline{x}\in\mathcal{D}$ such that $\overline{\sigma}=\sigma_{\max}(\nabla f(\overline{x}))=\underset{x\in\mathcal{D}}{\sup}\ \sigma_{\max}(\nabla f(x))<\infty$.  Then for $\forall x_k\in \mathcal{D}$, we have
			\begin{equation*}
				\sigma_{max}(\nabla f(x_k))\leq\underset{x\in\mathcal{D}}{\sup}\ \sigma_{\max}(\nabla f(x))=\overline{\sigma}<\infty.
			\end{equation*}
		\end{proof}
	\end{lem}
	
	\begin{thm}\label{th1}
		Assume that the Assumption \ref{asm1} holds. Let $\{x_k\}$ be generated by Algorithm \ref{alg: SCBGD}. If $0<\delta<\min\{2,\frac{4\underline{\sigma}^2}{L_{\max}}\}$ we have
		\begin{equation*}
			\mathbf{E}[\|g(x_k)\|_2^2]-\|g(x_*)\|_2^2\leq\frac{2\tau\underline{\sigma}^2\overline{\sigma}^2R_0^2}{k\delta(4\underline{\sigma}^2-\delta L_{\max})}.
		\end{equation*} 
		Furthermore, if $g(x)$ is strongly convex with $\gamma>0$ and $\delta^2 L_{\max}\gamma-4\delta\gamma\underline{\sigma}^2+\tau\overline{\sigma}^2\underline{\sigma}^2\geq0$, we can obtain
		\begin{equation*}
			\mathbf{E}[\|g(x_k)\|_2^2]-\|g(x_*)\|_2^2\leq\left(1-\frac{2\delta\gamma}{\tau\overline{\sigma}^2}\left(2-\frac{\delta L_{\max}}{2\underline{\sigma}^2}\right)\right)^k\left(\|g(x_0)\|_2^2-\|g(x_*)\|_2^2\right).
		\end{equation*}
	\end{thm}
	\begin{proof}
		Let $g(x_k)=\frac{1}{2}f(x_{k})^Tf(x_{k})$. Using (\ref{eq:g'}) and Lemma \ref{lem:g(v)}, we can obtain
		\begin{align*}
			g(x_{k+1})
			&=g[x_k-\delta\frac{\|\nabla f_{:,\xi_k}(x_k)^Tf(x_k)\|_2^2}{\|\nabla f_{:,\xi_k}(x_k)\nabla f_{:,\xi_k}(x_k)^Tf(x_k)\|_2^2}I_{:,\xi_k}\nabla f_{:,\xi_k}(x_k)^Tf(x_k)]\\
			&\leq g(x_k)-2\delta\left(\nabla f(x_k)^Tf(x_k)\right)^T\frac{\|\nabla f_{:,\xi_k}(x_k)^Tf(x_k)\|_2^2}{\|\nabla f_{:,\xi_k}(x_k)\nabla f_{:,\xi_k}(x_k)^Tf(x_k)\|_2^2}I_{:,\xi_k}\nabla f_{:,\xi_k}(x_k)^Tf(x_k)+\frac{1}{2}\delta^2L_{i_k}\frac{\|\nabla f_{:,\xi_k}(x_k)^Tf(x_k)\|_2^6}{\|\nabla f_{:,\xi_k}(x_k)\nabla f_{:,\xi_k}(x_k)^Tf(x_k)\|_2^4}\\
			&=g(x_k)-2\delta\frac{\|\nabla f_{:,\xi_k}(x_k)^Tf(x_k)\|_2^4}{\|\nabla f_{:,\xi_k}(x_k)\nabla f_{:,\xi_k}(x_k)^Tf(x_k)\|_2^2}+\frac{1}{2}\delta^2L_{i_k}\frac{\|\nabla f_{:,\xi_k}(x_k)^Tf(x_k)\|_2^6}{\|\nabla f_{:,\xi_k}(x_k)\nabla f_{:,\xi_k}(x_k)^Tf(x_k)\|_2^4}\\
			&\leq g(x_k)-\delta\left(2-\frac{\delta L_{i_k}}{2\sigma_{\min}^2(\nabla f_{:,\xi_k}(x_k))}\right)\frac{\|\nabla f_{:,\xi_k}(x_k)^Tf(x_k)\|_2^4}{\|\nabla f_{:,\xi_k}(x_k)\nabla f_{:,\xi_k}(x_k)^Tf(x_k)\|_2^2},\\
		\end{align*}
		where the last inequality makes use of $\nabla f_{:,\xi_k}(x_k)^Tf(x_k)\in \mathcal{R}(\nabla f_{:,\xi_k}(x_k)^T)$.
		Taking the conditional expectation on the above inequality, we obtain
		\begin{align}\label{E[g(xk+1)]<g(xk)-a}
			\mathbf{E}_k[g(x_{k+1})]
			&\leq g(x_k)-\frac{1}{\tau}\sum_{\xi_k}\delta\left(2-\frac{\delta L_{i_k}}{2\sigma_{\min}^2(\nabla f_{:,\xi_k}(x_k))}\right)\frac{\|\nabla f_{:,\xi_k}(x_k)^Tf(x_k)\|_2^4}{\|\nabla f_{:,\xi_k}(x_k)\nabla f_{:,\xi_k}(x_k)^Tf(x_k)\|_2^2}\nonumber\\
			&\leq g(x_k)-\frac{1}{\tau}\sum_{\xi_k}\delta\left(2-\frac{\delta L_{\max}}{2\sigma_{\min}^2(\nabla f_{:,\xi_k}(x_k))}\right)\frac{\|\nabla f_{:,\xi_k}(x_k)^Tf(x_k)\|_2^4}{\sigma_{\max}^2(\nabla f_{:,\xi_k}(x_k))\|\nabla f_{:,\xi_k}(x_k)^Tf(x_k)\|_2^2}\nonumber\\
			&\leq g(x_k)-\frac{\delta}{\tau\overline{\sigma}^2}\left(2-\frac{\delta L_{\max}}{2\underline{\sigma}^2}\right)\sum_{\xi_k}\|\nabla f_{:,\xi_k}(x_k)^Tf(x_k)\|_2^2\nonumber\\
			&=g(x_k)-\frac{\delta}{\tau\overline{\sigma}^2}\left(2-\frac{\delta L_{\max}}{2\underline{\sigma}^2}\right)\|\nabla g(x_k)\|_2^2,
		\end{align}
		where the last inequality comes from the Lemma \ref{lem:sigma}.
		Then we obtain
		\begin{equation*}
			\mathbf{E}_k[g(x_{k+1})]-g(x_*)\leq g(x_k)-g(x_*)-\frac{\delta}{\tau\overline{\sigma}^2}\left(2-\frac{\delta L_{\max}}{2\underline{\sigma}^2}\right)\|\nabla g(x_k)\|_2^2.
		\end{equation*}
		Let $\psi_k=\mathbf{E}_{k-1}[g(x_k)]-g(x_*)$, $\alpha=\frac{\delta}{\tau\overline{\sigma}^2}\left(2-\frac{\delta L_{\max}}{2\underline{\sigma}^2}\right)$, then we can get 
		\begin{equation}\label{eq:psi}
			\psi_{k+1}\leq\psi_k-\alpha\mathbf{E}_{k-1}[\|\nabla g(x_k)\|_2^2]\leq\psi_k-\alpha(\mathbf{E}_{k-1}[\|\nabla g(x_k)\|_2])^2.
		\end{equation}
		Since $\delta\in(0,\min\{2,\frac{4\underline{\sigma}^2}{L_{\max}}\})$, we know $\left(2-\frac{\delta L_{\max}}{2\underline{\sigma}^2}\right)>0$. Combining with (\ref{E[g(xk+1)]<g(xk)-a}), we can get $x_k$ in the level set $\Omega$ by induction. Then we have $\|x_k-x_*\|_2^2\leq R_0$. By convexity of $g(x)$, it holds
		\begin{equation*}
			g(x_k)-g(x_*)\leq\nabla g(x_k)^T(x_k-x_*)\leq\|\nabla g(x_k)\|_2\|x_k-x_*\|_2\leq R_0\|\nabla g(x_k)\|_2.
		\end{equation*} 
		Taking the conditional expectation on the above inequality, we can obtain $\mathbf{E}_{k-1}[\|\nabla g(x_k)\|_2]\geq\frac{1}{R_0}\psi_k$.
		Substituting the above equality into (\ref{eq:psi}), we have $\psi_{k}-\psi_{k+1}\geq\frac{\alpha}{R_0^2}\psi_k^2$.
		Then $\frac{1}{\psi_{k+1}}-\frac{1}{\psi_k}\geq\frac{\alpha}{R_0^2}$.
		By applying the above inequality recursively, we have $\frac{1}{\psi_k}\geq\frac{1}{\psi_0}+\frac{k\alpha}{R_0^2}\geq\frac{k\alpha}{R_0^2}$.
		Then we can get
		\begin{equation*}
			\psi_k=\mathbf{E}_{k-1}[g(x_k)]-g(x_*)\leq\frac{R_0^2}{k\alpha}\leq\frac{2\tau\underline{\sigma}^2\overline{\sigma}^2R_0^2}{k\delta(4\underline{\sigma}^2-\delta L_{\max})}.
		\end{equation*}
		Furthermore, when $g(x)$ is strongly convex with $\gamma$, that is 
		\begin{equation*}
			g(v)\geq g(u)+\nabla g(u)^T(v-u)+\frac{\gamma}{2}\|v-u\|_2^2,\ for\ all\ u, v.
		\end{equation*} 
		Taking the minimum of both sides with respect to $v$ and setting $u=x_k$, we get
		\begin{equation*}
			g(x_*)\geq g(x_k)-\frac{1}{2\gamma}\|\nabla g(x_k)\|_2^2.
		\end{equation*} 
		Taking the conditional expectation on the above inequality, we have
		\begin{equation*}
			g(x_*)\geq\mathbf{E}_{k-1}[g(x_k)]-\frac{1}{2\gamma}\mathbf{E}_{k-1}[\|\nabla g(x_k)\|_2^2]\\
		\end{equation*}		
		\begin{equation*}
			\mathbf{E}_{k-1}[\|\nabla g(x_k)\|_2^2]\geq2\gamma(\mathbf{E}_{k-1}[g(x_k)]-g(x_*)).
		\end{equation*}
		Combining with (\ref{eq:psi}), we can obtain $\psi_{k+1}\leq\psi_k-2\gamma\alpha\psi_k=(1-2\gamma\alpha)\psi_k$.
	\end{proof}
	
	\begin{cor}
		If $g(x)$ is not convex but satisfies the PL condition \cite{Loizou20} i.e. if there exists $\mu>0$ such that : $\|\nabla g(x)\|_2^2\geq2\mu (g(x)-g(x_*))$. Let $\{x_k\}$ be generated by Algorithm \ref{alg: SCBGD}. If  $0<\delta<\min\{2,\frac{4\underline{\sigma}^2}{L_{\max}}\}$ we have 
		\begin{equation*}
			\mathbf{E}[\|g(x_k)\|_2^2]-\|g(x_*)\|_2^2\leq\left(1-\frac{2\delta\mu}{\tau\overline{\sigma}^2}\left(2-\frac{\delta L_{\max}}{2\underline{\sigma}^2}\right)\right)^k\left(\|g(x_0)\|_2^2-\|g(x_*)\|_2^2\right).
		\end{equation*}
	\end{cor}
	
	\section{Numerical experiments}
	\label{sec3}
	In this section, we present some numerical experiments to verify the effectiveness of the SCBGD method for solving nonlinear systems of equations. We compare the SCBGD method with GD, SGD and SCBN methods on different nonlinear problems. We use the number of iteration steps (IT) and the elapsed computing time in seconds (CPU) to measure the effectiveness of these methods. The IT and CPU are obtained from the average value obtained from 10 experiments. The iteration termination criterion is $\|f(x_k)\|_2\leq10^{-6}$ or the IT $\geq 200000$ .We use '--' to indicate that the method does not converge.
	All experiments are performed using MATLAB (version R2021a) on a personal computer with a 2.20 GHz central processing unit  (13th Gen Intel (R) Core (TM) i9-13900HX), 32.0 GB memory and Windows operating system (64-bit Windows 11).
	
	\begin{exm}
	Broyden tridiagonal problem\cite{Luksan18}\label{ex:Broyden tridiagonal}.
	\begin{equation*}
		f_k(x)=(0.5x_k-3)x_k+x_{k-1}+2x_{k+1}-1=0, \text{where } x_0=0 \text{ and } x_{n+1}=0
	\end{equation*}

	\end{exm}
	In this experiment, we set the initial estimate $x_0=(-1.5,-1.5,...,-1.5)^T\in \mathbb{R}^n$ and $n=200:200:1000$. The corresponding results of these methods are depicted in Table \ref{tab:tridiagonal}, which also shows the IT and CPU of these methods. It shows that the SCBGD method requires less computing time than other methods. In Figure \ref{figure tridiagonal}, we find SCBGD method has a significantly better performance than other methods.

	\begin{table}
	\centering
	\caption{IT and CPU of the GD, SGD, SCBN  and SCBGD methods for Broyden tridiagonal problem}
	\label{tab:tridiagonal}
	\vspace{1mm}
	\resizebox{0.55\linewidth}{1.9cm}{
		\begin{tabular}{cclllll}
			\toprule
			Method&  $n$& $200$& $400$& $600$& $800$& $1000$\\
			\midrule
			GD&
			IT & $201$& $203$&$205$&$206$&$208$\\
			&CPU & $\mathbf{0.0178}$& $0.1183$&$0.2211$&$0.3483$&$0.4965$\\
			SGD&IT
			& $3653$& $7455$&$11306$&$15177$&$19118$\\
			($q=10$)&CPU & $0.0201$& $0.0659$&$0.1429$&$0.2917$&$0.4423$\\
			SCBN&IT
			& $352$& $721$&$1102$&$1480$&$1862$\\
			($q=100$, $\alpha=100$)&CPU & $0.0191$& $0.0781$&$0.1819$&$0.3187$&$0.4123$\\
			SCBGD&IT
			& $3509$& $7220$&$10963$&$14783$&$18612$\\
			($q=10$, $\delta=1$)&CPU & $0.0181$& $\mathbf{0.0522}$&$\mathbf{0.1010}$&$\mathbf{0.1711}$&$\mathbf{0.2420}$\\\bottomrule\end{tabular}}
    \end{table}
	\begin{figure}[!t]
		\centering
		\subfigure[$n=600$]{
			\includegraphics[scale=0.55]{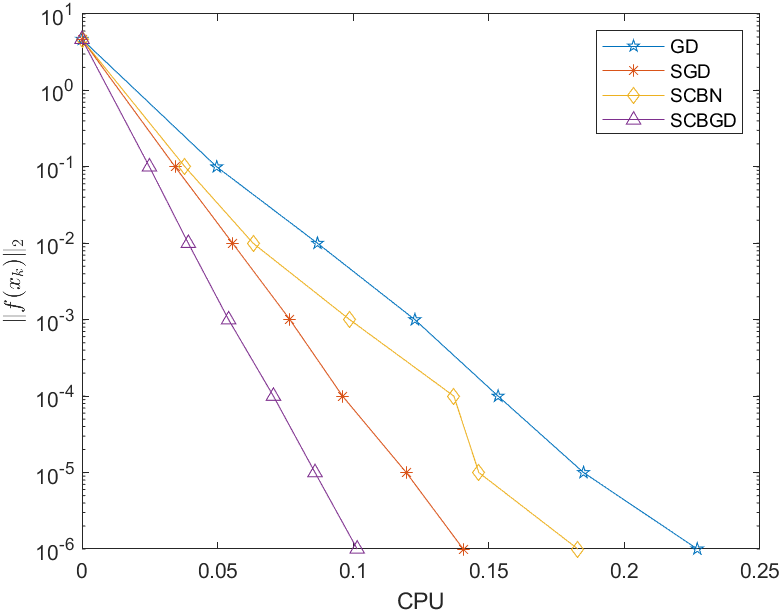}}
		\subfigure[$n=1000$]{
			\includegraphics[scale=0.55]{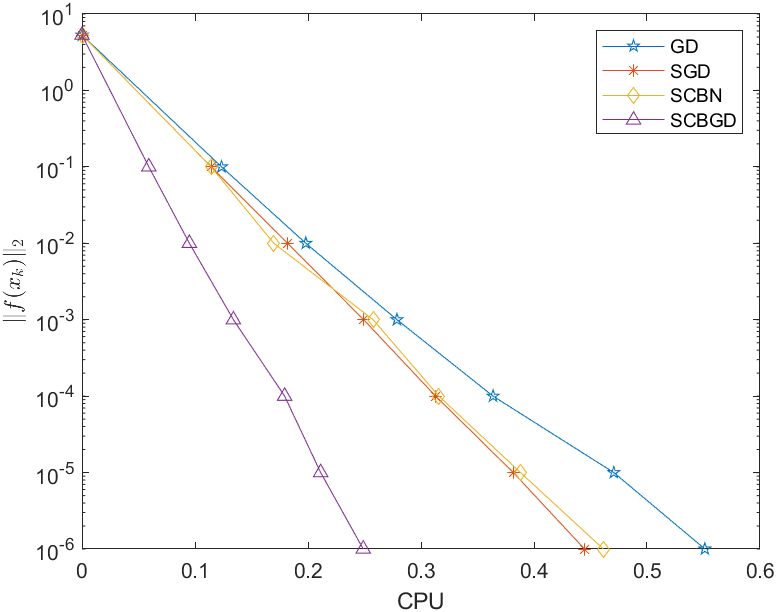}}\\
		\caption[d]{CPU versus residual for Broyden tridiagonal problem }
		\label{figure tridiagonal}
	\end{figure}

	\begin{exm}
		Tridiagonal system \cite{Li89}
		\begin{align*}
			&f_k(x)=4(x_k-x_{k+1}^2)=0, &k=1,\\
			&f_k(x)=8x_k(x_k^2-x_{k-1})-2(1-x_k)+4(x_k-x_{k+1}^2)=0, &1<k<n,\\
			&f_k(x)=8x_k(x_k^2-x_{k-1})-2(1-x_k)=0, &k=n.
		\end{align*}
	\end{exm}
	In this experiment, we set the initial estimate $x_0=(0.5,0.5,...,0.5)^T\in \mathbb{R}^n$ and $n=200:200:1000$. The corresponding results of these methods are shown in Table \ref{tab:Tridiagonal} and Figure \ref{figure Tridiagonal}. Since the SGD method does not convergence for $n>=600$, we remove the image of the SGD method in Figure \ref{figure Tridiagonal}. Table \ref{tab:Tridiagonal} and Figure \ref{figure Tridiagonal} indicate that for big size system the SCBGD method requires less CPU time than other methods, demonstrating that the SCBGD method outperforms the others.
	
	\begin{table}
		\centering
		\caption{IT and CPU of the GD, SGD, SCBN and SCBGD methods for Tridiagonal system}
		\label{tab:Tridiagonal}
		\vspace{1mm}
		\resizebox{0.55\linewidth}{1.9cm}{
			\begin{tabular}{cclllll}
					\toprule
					Method&  $n$& $200$& $400$& $600$& $800$& $1000$\\
					\midrule
					GD&
					IT & $15507$& $15680$&$15771$&$15825$&$15861$\\
					&CPU & $1.8744$& $10.4756$&$17.8661$&$27.8024$&$42.1529$\\
					SGD&IT
					& $11079$& $15719$&$--$&$--$&$--$\\
					($q=100$)&CPU & $\mathbf{0.5598}$& $2.3530$&$--$&$--$&$--$\\
					SCBN&IT
					& $28549$& $60221$&$86223$&$114276$&$144472$\\
					($q=100$, $\alpha=105$)&CPU & $1.4584$& $5.6616$&$16.1257$&$29.2071$&$43.4630$\\
					SCBGD&IT
					& $11953$& $15215$&$17486$&$28128$&$30515$\\
					($q=100$, $\delta=1$)&CPU & $0.6111$& $\mathbf{2.0744}$&$\mathbf{3.8972}$&$\mathbf{7.3346}$&$\mathbf{10.0678}$\\\bottomrule\end{tabular}}
			\end{table}

		\begin{figure}[!t]
		\centering
		\subfigure[$n=600$]{
				\includegraphics[scale=0.55]{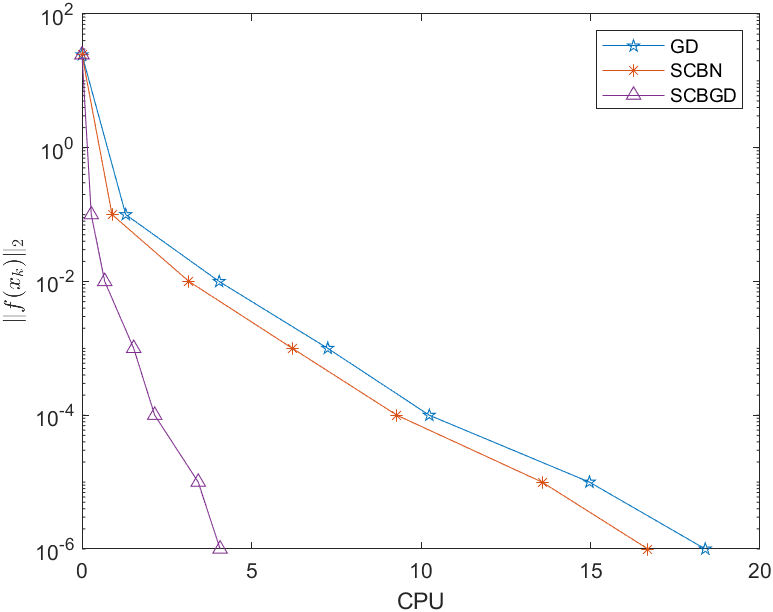}}
		\subfigure[$n=1000$]{
				\includegraphics[scale=0.55]{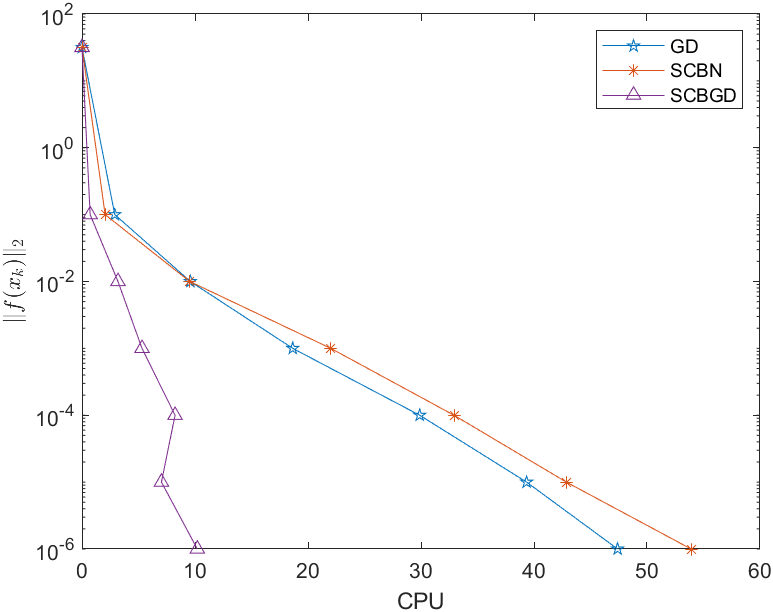}}\\
		\caption[d]{CPU versus residual for Tridiagonal system}
		\label{figure Tridiagonal}
		\end{figure}
	
	\section{Conclusion}
	\label{sec4}
	In this paper, we introduced the idea of column block into the GD method and proposed the SCBGD method with a uniform probability for solving nonlinear systems of equations. Under certain conditions, we proved the convergence of the SCBGD method and derived the upper bound for the convergence rate of the new method. In the numerical experiments, the CPU time of the SCBGD method is less than that of the GD, SGD, and SCBN methods, which demonstrates the efficiency of the new method. In future work, we will investigate strategies for choosing block sizes.

	\section*{Use of AI tools declaration}
	The authors declare they have not used Artificial Intelligence (AI) tools in the creation of this article.
		\section*{ Funding}
	This work was supported by the National Natural Science Foundation
	of China [grant numbers  61931025, 42374156]
	\small
	
	\setlength{\bibspacing}{0\baselineskip}
	\bibliographystyle{model1-num-names}
	\bibliography{reference1}
	
\end{document}